\documentclass [12pt]{amsart}
\usepackage{amsmath,amssymb}
 \usepackage{amsthm, amsfonts}
 \usepackage{enumerate}

\newtheorem{theorem}{Theorem}[section]

\newtheorem{lemma}[theorem]{Lemma}
\newtheorem{corollary}[theorem]{Corollary}
\theoremstyle{definition}
\newtheorem{definition}[theorem]{Definition}

\theoremstyle{remark}

\numberwithin{equation}{section}

\begin{document}

\title [{{On graded primary-like submodules}}]{On graded primary-like submodules of graded modules over graded commutative rings}

 \author[{{K. Al-Zoubi and M. Al-Dolat }}]{\textit{Khaldoun Al-Zoubi{*} and Mohammed Al-Dolat   }}

\address
{\textit{Khaldoun Al-Zoubi, Department of Mathematics and
Statistics, Jordan University of Science and Technology, P.O.Box
3030, Irbid 22110, Jordan.}}
\bigskip
{\email{\textit{kfzoubi@just.edu.jo}}}

\address
{\textit{Mohammed Al-Dolat, Department of Mathematics and
Statistics, Jordan University of Science and Technology, P.O.Box
3030, Irbid 22110, Jordan.}}
\bigskip
{\email{\textit{mmaldolat@just.edu.jo }}}


\date{}

\begin{abstract}
Let $G$ be a group with identity $e$. Let $R$ be a $G$-graded
commutative ring and $M$ a graded $R$-module. In this paper, we
introduce the concept of graded primary-like submodules as a new
generalization of graded primary ideals and give some basic results
about graded primary-like submodules of graded modules. Special
attention has been paid, when graded submodules satisfies the
gr-primeful property, to find extra properties of these graded
submodules.

\end{abstract}

\keywords{graded primary ideals, graded primary-like submodules,
graded prime submodules, gr-primeful property.
  \\$*$ Corresponding author
 }
 \maketitle


 \section{Introduction}
     Recently, H. F. Moghimi and F. Rashedi, in \cite{18} studied primary-like submodules as a new generalization of primary ideals to modules. Also, the concept of primeful module was introduced and studied by C.P. Lu in \cite{16}.

  The scope of this paper is devoted to the theory of graded modules
over graded commutative rings. One use of rings and modules
with gradings is in describing certain topics in algebraic geometry.
Here, in particular, we are dealing with graded primary-like
submodules.

The concept of graded primary ideal was introduced and studied by M.
Refai and K. Al-Zoubi in \cite{23}.

In the literature, there are several different generalization of the
notion of graded primary ideal to graded module. The concept of
graded primary submodule was introduced by S.E. Atani and F.
Farzalipour in \cite{12} and studied in \cite{1, 4, 14, 22}. Also
the the concept of graded prime submodule was introduced by S.E.
Atani in \cite{7} and studied in \cite{2, 3, 5, 6, 9, 10, 17, 22, 25}.

Here, we introduce the concept of graded primary-like submodule as a
new generalization of a graded primary ideal on the one hand and a
generalization of a graded prime submodule on other hand.

Our article is organized as follows.

In Section 2 we recall important notions which will be used
throughout the paper. In Section 3 we will investigate graded
submodules which satisfy the gr-primeful property. In Section 4 we
introduce the concept of graded primary-like submodules and give a
number of results concerning such modules. For example, we give a
characterization of graded primary-like submodules. We also study
the behavior of graded primary-like submodules under graded
homomorphisms and under localization.


 \section{Preliminaries}
\textbf{Convention}. Throughout this paper all rings are commutative
with identity and all modules are unitary.

First, we recall some basic properties of graded rings and modules
which will be used in the sequel. We refer to \cite{15}, \cite{19},
\cite{20} and \cite{21} for these basic properties and more
information on graded rings and modules.

Let $G$ be a group with identity $e$ and $R$ be a commutative ring
with identity $1_{R}$. Then $R$ is a \emph{$G$-graded ring} if there
exist additive subgroups $R_{g}$ of $R$ such that $R=\bigoplus_{g\in G}R_{g}$ and $%
R_{g}R_{h}\subseteq R_{gh}$ for all $g,h\in G$. The elements of
$R_{g}$ are called to be \emph{homogeneous} of degree $g$ where the
$R_{g}$'s are additive subgroups of $R$ indexed by the elements
$g\in G$. If $x\in R$, then $x$ can be written uniquely as
$\sum_{g\in G}x_{g}$, where $x_{g}$ is the component of $x$ in
$R_{g}$. Moreover, $h(R)=\bigcup_{g\in G}R_{g}$. Let $I$ be an ideal
of $R$. Then $I$ is called a \emph{graded ideal} of $(R,G)$
if $I=\bigoplus_{g\in G}(I\bigcap R_{g})$. Thus, if $x\in I$, then $%
x=\sum_{g\in G}x_{g}$ with $x_{g}\in I$. An ideal of a $G$-graded
ring need not be $G$-graded.

Let $R$ be a $G$-graded ring and $M$ an $R$-module. We say that $M$
is a \emph{$G$-graded $R$-module} (or \emph{graded $R$-module}) if
there exists a family of subgroups $\{M_{g}\}_{_{g\in G}}$ of $M$ such that $M=\underset{%
g\in G}{\bigoplus }M_{g}$ (as abelian groups) and
$R_{g}M_{h}\subseteq M_{gh} $ for all $g,h\in G$. Here, $R_{g}M_{h}$
denotes the additive subgroup of $M$ consisting of all finite sums of elements $r_{g}s_{h}$ with $%
r_{g}\in R_{g}$ and $s_{h}\in M_{h}.$ Also, we write $h(M)=\underset{g\in G}{%
\bigcup }M_{g}$ and the elements of $h(M)$ are called to be \emph{homogeneous%
}. Let $M=\underset{g\in G}{\bigoplus }M_{g}$ be a graded $R$-module
and $N$ a submodule of $M$. Then $N$ is called a \emph{graded submodule} of $M$ if $%
N=\underset{g\in G}{\bigoplus }N_{g}$ where $N_{g}=N\cap M_{g}$ for
$g\in G.$ In this case, $N_{g}$ is called the \emph{$g$-component}
of $N$.

Let $R$ be a $G$-graded ring and $S\subseteq h(R)$ be a
multiplicatively closed subset of $R.$ Then the ring of fraction
$S^{-1}R$ is a graded ring which is called the graded ring of fractions. Indeed, $%
S^{-1}R$ $=\underset{g\in G}{\oplus }(S^{-1}R)_{g}$ where $%
(S^{-1}R)_{g}=\{r/s:r\in R,s\in S$ $\ and$ $g=(\deg s)^{-1}(\deg r)\}.$ Let $%
M$ be a graded module over a $G$-graded ring $R$ and  $S\subseteq
h(R)$ be a multiplicatively closed subset of $R$. The module of
fraction  $S^{-1}M$ over a graded ring $S^{-1}R$ is a graded module
which is called module of fractions, if $\ S^{-1}M=\underset{g\in G}{\oplus }(S^{-1}M)_{g}$ where $%
(S^{-1}M)_{g}=\{m/s:m\in M,s\in S$ $\ and$ $g=(\deg s)^{-1}(\deg
m)\}.$ We write $h(S^{-1}R)=\underset{g\in G}{\cup }(S^{-1}R)_{g}$ and  $h(S^{-1}M)=%
\underset{g\in G}{\cup }(S^{-1}M)_{g}$. Consider the graded homomorphism $\eta :M\rightarrow S^{-1}M$ defined by $%
\eta (m)=m/1.$ For any graded submodule $N$ of $M,$ the submodule of
$S^{-1}M $ generated by $\eta (N)$ is denoted by $S^{-1}N.$ Similar
to non graded case, one can prove that $S^{-1}N=\left\{ \beta \in
S^{-1}M:\beta =m/s\text{ for }m\in N\text{ and }s\in S\right\} $ and
that $S^{-1}N\neq S^{-1}M$ if and only if $S\cap (N:_{R}M)=\phi .$
If $K$ is a graded submodule of $S^{-1}R $-module $S^{-1}M,$ then
$K\cap M$ will denote the graded submodule $\eta ^{-1}(K)$ of $M.$
Moreover, similar to the non graded case one can prove that
$S^{-1}(K\cap M)=K$.

Let $R$ be a $G$-graded ring and $M$ a graded $R$-module.\\ A proper
graded ideal $I$ of $R$ is said to be \textit{a graded maximal
ideal} of $R$ if $J$ is a graded ideal of $R$ such that $I\subseteq
J\subseteq R$, then $I=J$ or $J=R$ (see\cite{24}.)

A proper graded ideal $I$ of $R$ is said to be \textit{a graded
prime ideal }if whenever $rs\in I$, we have $r\in I$ or $s\in I,$
where $r,s\in h(R)$ (see \cite {24}.)
The \textit{graded radical} of $I$, denoted by Gr$(I)$, is the set of all $x=\sum\nolimits_{g\in G}x_{g} \in R$ such that for each $g\in G$ there exists $n_{g} \in \mathbb{Z}^{+}$ with $%
x^{n_{g}}\in I$. Note that, if $r$ is a homogeneous element, then
$r\in $Gr$(I) $ if and only if $r^{n}\in I$ for some $n\in \mathbb{N}$
(see \cite {24}.) It is shown in \cite[Proposition 2.5]{24} that
Gr$(I)$ is the intersection of all graded prime ideals of $R$
containing $I$.\\
A proper graded ideal $P$ of $R$ is said to be \emph{a graded primary ideal} if whenever $%
r,s\in h(R)$ with $rs\in P$, then either $r\in P$ or $s\in $Gr$(P)$
(see \cite{23}.)

 A proper graded submodule $N$ of $M$ is said to be \textit{a
graded prime submodule} if whenever $r\in h(R)$ and $m\in h(M)$ with
$rm\in N$, then either $r\in (N:_{R}M)=\{r\in R:rM\subseteq N\}$\emph{\ or }$m\in N$%
\emph{\ }(see \cite{7}.) A proper graded submodule $N$ of a graded
$R$-module $M$ is said to be \textit{a graded primary submodule } if
whenever $r\in h(R)$ and $m\in h(M)$ with $rm\in N$, then either
$m\in N$ or $r\in $Gr$((N:_{_{R}}M))$
(see \cite{12}.)\\

The \textit{graded radical} of a graded submodule $N$ of $M$,
denoted by Gr$_{M}(N)$, is defined to be the intersection of all
graded prime submodules of $M$ containing $N$. If $N$ is not
contained in any graded prime submodule of $M$, then Gr$_{M}(N)=M$
(see \cite{12}.)

A graded $R$-module $M$ over $G$-graded ring $R$ is said to be
\textit{a graded multiplication module (gr-multiplication module)}
if for every graded submodule $N$ of $M$ there exists a graded ideal
$I$ of $R$ such that $N=IM.$ It is clear that $M$ is
$gr$-multiplication $R$- module if and only if $N=(N:_{R}M)M$ for
every graded submodule $N$ of $M$ (see \cite{13}.)


 \section{Graded submodules which satisfy the gr-primeful property}


The following Lemma is known (see \cite[Lemma 1.2 and Lemma
2.7]{17}), we write it her for the sake of
 references.
\begin{lemma}
Let $R$ be a G-graded ring and $M$ a graded R-module. Then the
following hold:
\begin{enumerate}

\item If $N$ is a graded submodule of $M$, then $(N:_{_{R}}M)=\{r\in R:rM\subseteq N\}$ is a
graded ideal of R.

\item  If $N$ is a graded submodule of $M$, $r\in h(R),$ $x\in h(M)$ and $I$ is a
graded ideal of $R$, then $Rx,$ $IN$ and $rN$ are graded submodules
of $M.$

\item If  $N$ and $K$ are graded submodules of $M,$ then  $N+K$ and
$N\cap K$ are also graded submodules of $M$.

\item If $\{N_{i}\}_{i\in I}$ is a collection of graded submodules of $M$, then $N=%
\underset{i\in I}{\cap }$ $N_{i}$ is  a graded submodule of $M$.

\end{enumerate}\

\end{lemma}

\begin{definition}
Let $N$ be a graded submodule of a graded $R$-module $M$. We say
that $N$ satisfies the gr-primeful property if for each graded
prime ideal $p$ of $R$ with $(N:_{R}M)\subseteq p,$ there exists a
graded prime submodule $P$ of $M$ containing $N$ such that
$(P:_{R}M)=p.$ A graded $R$-module $M$ is called gr-primeful, if
either $M=0$ or the zero graded submodule of $M$ satisfies the
gr-primeful property.
\end{definition}

\begin{theorem}
Let $R$ be a $G$-graded ring and $M$ a graded $R$-module. If a
graded submodule $N$ of $M$ satisfies the gr-primeful property,
then Gr$_{M}(N)$ satisfies the gr-primeful property.
\end{theorem}

\begin{proof}
Suppose that $p$ is a graded prime ideal of $R$ containing $%
($Gr$_{M}(N):_{R}M).$ Since $N$ satisfies the gr-primeful property and $%
(N:_{R}M)\subseteq ($Gr$_{M}(N):_{R}M)\subseteq p,$ there exists a
graded prime submodule $P$ of $M$ containing $N$ such that
$(P:_{R}M)=p$. It is clear that Gr$_{M}(N)\subseteq P.$ Thus
Gr$_{M}(N)$ satisfies the gr-primeful property.
\end{proof}

\begin{lemma}
Let $R$ be a $G$-graded ring, $M$ a graded $R$-module and $N$ a
graded submodule of $M$. Then Gr$((N:_{R}M))\subseteq
($Gr$_{M}(N):_{R}M).$
\end{lemma}
\begin{proof}
Let $P$ be a graded prime submodule of $M$ containing $N$. By
\cite[Proposition 2.7(ii)]{7}, $(P:_{R}M)$ is a graded prime ideal of $%
R.$ Since $N\subseteq P,$ we get $(N:_{R}M)\subseteq (P:_{R}M).$ By
\cite[Proposition 1.2]{23}, we conclude that
Gr$((N:_{R}M))\subseteq (P:_{R}M).$ Thus Gr$((N:_{R}M))M\subseteq
(P:_{R}M)M\subseteq P$ this implies that Gr$((N:_{R}M))M\subseteq
$Gr$_{M}(N).$ Therefore Gr$((N:_{R}M))\subseteq ($Gr$_{M}(N):_{R}M).$
\end{proof}

\begin{theorem}
Let $R$ be a $G$-graded ring and $M$ a graded $R$-module. If $N$ is
a graded submodule of $M$ satisfying the gr-primeful property, then
$($Gr$_{M}(N):_{R}M)=Gr((N:_{R}M)).$
\end{theorem}
\begin{proof}
Let $p$ be a graded prime ideal of $R$ containing $(N:_{R}M).$ Since
$N$ satisfies the gr-primeful property, there exists a graded prime submodule $%
P$ of $M$ containing $N$ such that \ $(P:_{R}M)=p.$ This implies that
$($Gr$_{M}(N):_{R}M)\subseteq (P:_{R}M)\subseteq p$ and hence
$($Gr$_{M}(N):_{R}M)\subseteq $Gr$((N:_{R}M)).$ By Lemma 3.4,
Gr$((N:_{R}M))\subseteq ($Gr$_{M}(N):_{R}M).$ Therefore $%
($Gr$_{M}(N):_{R}M)=$Gr$((N:_{R}M)).$
\end{proof}

Let $R$ be a $G$-graded ring and $M$, $M^{\prime }$ graded $R$-modules. Let $%
\varphi :M\rightarrow M^{\prime }$ be an $R$-module homomorphism. Then $%
\varphi $ is said to be a graded homomorphism if $\varphi
(M_{g})\subseteq M_{g}^{\prime }$ for all $g\in G$ (see \cite{19}.)

\begin{theorem}
Let $R$ be a $G$-graded ring and $M,$ $M^{\prime }$ be two graded
$R$-modules and $N^{\prime }$ a graded submodule of $M^{\prime }.$
Let $\varphi :M\rightarrow M^{\prime }$ be a graded epimorphism. If
$N^{\prime }$ satisfies the gr-primeful property, then so does
$\varphi ^{-1}($ $N^{\prime }).$
\end{theorem}
\begin{proof}
Let $p$ be a graded prime ideal of $R$ such that $(\varphi ^{-1}($
$N^{\prime }):_{R}M)\subseteq p.$ We show that $(N^{\prime
}:_{R}M^{\prime })\subseteq p.$ Let $r\in (N^{\prime }:_{R}M^{\prime
})\cap
h(R).$ Using the fact that $\varphi $ is a graded epimorphism, we have $%
rM^{\prime }=r\varphi (M)=\varphi (rM)\subseteq N^{\prime }$ and so $%
rM\subseteq \varphi ^{-1}(N^{\prime }).$ Hence $r\in (\varphi ^{-1}($ $%
N^{\prime }):_{R}M)\subseteq p.$ Thus $(N^{\prime }:_{R}M^{\prime
})\subseteq p.$ Since $N^{\prime }$ satisfies the gr-primeful
property, there exists a graded prime submodule $P^{\prime }$ of
$M^{\prime }$ containing $N^{\prime } $ such that $(P^{\prime
}:_{R}M^{\prime })=p.$ By \cite[Lemma 5.2(ii)]{11}, we have that
 $\varphi ^{-1}($ $P^{\prime })$ is a graded prime submodule of $M$
containing $\varphi ^{-1}($ $N^{\prime }).$  It is easy to see
$(\varphi ^{-1}(P^{\prime }):_{R}M)=p.$ Therefore $\varphi ^{-1}($
$N^{\prime })$ satisfies the gr-primeful property.
\end{proof}
\begin{theorem}
Let $R$ be a $G$-graded ring and $M,$ $M^{\prime }$ be two graded
$R$-modules. Let $\varphi :M\rightarrow M^{\prime }$ be a graded
epimorphism and $N$ a graded submodule of $M$ containing Ker$\varphi .$ If $N$ satisfies the
gr-primeful property, then $\varphi (N)\ $satisfies the
gr-primeful property.
\end{theorem}
\begin{proof}
Let $p$ be a graded prime ideal of $R$ such that $(\ \varphi (N)\
:_{R}M^{\prime })\subseteq p.$ We show that $(N:_{R}M)\subseteq p.$
Let $r\in (N:_{R}M)\cap h(R)$ so $rM\subseteq N.$ Using the fact
that $\varphi $ is a graded epimorphism, we have $r\varphi
(M)=rM^{\prime }\subseteq \varphi ($ $N).$ Hence $r\in (\ \varphi
(N)\ :_{R}M^{\prime })\subseteq p.$ Thus $(N:_{R}M)\subseteq p.$
Since $N^{\ }$ satisfies the gr-primeful property, there exists a
graded prime submodule $P$ of $M$ containing $N$ such that $(P:_{R}M
)=p.$ Since Ker$\varphi \subseteq P$ by \cite[Lemma 5.2(i)]{11}, we
have $\varphi (P)$ is a graded prime submodule of $M^{\prime }$
containing $\varphi (N).$ It is easy to see $(\varphi
(P):_{R}M^{\prime })=p.$ Therefore $\varphi (N)$ satisfies the
gr-primeful property.
\end{proof}

\begin{theorem}
Let $R$ be a $G$-graded ring, $M$ a graded $R$-module and
$\{N_{i}:1\leq i\leq n\}$ be a finite collection of graded
submodules of $M$ satisfying the gr-primeful property. Then $\cap _{i=1}^{n}N_{i}$ satisfies the gr-primeful property.
\end{theorem}
\begin{proof}
Let $p$ be a graded prime ideal of $R$ such that $(\cap
_{i=1}^{n}N_{i}:_{R}M)\subseteq p$. By \cite[Lemma 2.6]{6}, we have
$(\cap _{i=1}^{n}N_{i}:_{R}M)=\cap _{i=1}^{n}(N_{i}:_{R}M).$ Since
$\cap _{i=1}^{n}(N_{i}:_{R}M)\subseteq p$ by \cite[Proposition
1.4]{23}, we have $(N_{j}:_{R}M)\subseteq p$ for some $j\in
\{1,2,...,n\}.$ Since $N_{j}$ satisfies the gr-primeful property,
there exists a graded prime submodule $P$ of $M$ containing $N_{j}$
with $(P:_{R}M)=p$ and so $\cap
_{i=1}^{n}N_{i}\subseteq P.$ Thus $\cap _{i=1}^{n}N_{i}$ satisfies the gr-primeful property.
\end{proof}
 \section{Gragded primary-like submodules}


\begin{definition}
Let $R$ be a $G$-graded ring and $M$ a graded $R$%
-module. A proper graded submodule $N$ of $M$ is said to be
\textit{a graded primary-like submodule} if whenever $r\in h(R)$ and
$m\in h(M)$ with $rm\in N$, then either $r\in (N:_{R}M)$ or $m\in
$Gr$_{M}(N).$
\end{definition}
\begin{theorem}
Let $R$ be a $G$-graded ring and $M $ a graded $R$-module. If $N$ is
a graded primary-like submodule of $M$ satisfying the gr-primeful
property, then $(N:_{R}M)$ is a graded primary ideal of $R$.
\end{theorem}
\begin{proof}
Suppose that $rs\in (N:_{R}M)$ and $s\notin (N:_{R}M)$ for some
$r,s\in h(R).$ We show that $r\in ($Gr$(N):_{R}M).$ Let $m=\sum_{g\in
G}m_{g}\in M.$ Hence $\sum_{g\in G}rsm_{g}\in N.$ Since $N$ is a
graded primary-like submodule of $M$, $rsm_{g}\in N$ and $s\notin
(N:_{R}M)$ for all $g\in G,$ we conclude that $rm_{g}\in $Gr$_{M}(N)$
for all $g\in G.$ Hence $rm\in
$Gr$_{M}(N).$ This shows that $r\in ($Gr$(N):M).$ Since $N$ satisfies the gr-primeful property, by Theorem 3.5, we have $r\in
($Gr$(N):_{R}M)=$Gr$((N:_{R}M)).$ Therefore $(N:_{R}M)$ is a graded
primary ideal of $R$.
\end{proof}


By \cite[Lemma 1.8]{23} and Theorem 4.2, we have the following
Corollary.
\begin{corollary}
Let $R$ be a $G$-graded ring and $M$ a graded $R$-module. If $N$ is
a graded primary-like submodule of $M$ satisfying the gr-primeful property, then $%
p=$Gr$((N:_{R}M))$ is a graded prime ideal of $R$ and we say that $N$
is a graded $p$-primary-like submodule of $M$.
\end{corollary}


\begin{theorem}
Let $R$ be a $G$-graded ring, $M$ a faithful graded multiplication $R$%
-module and $Q$ a graded primary deal of $R$. If $rm \in \ QM$ for
$r\in h(R),$ $m\in h(M),$ then $r\in Q$ or $m\in $Gr$_{M}(QM).$
\end{theorem}
\begin{proof}
The proof is similar to the proof of \cite[Theorem 2.2]{14} and
\cite[Lemma 2]{22},  so we omit it.
\end{proof}
\begin{corollary}
Let $R$ be a $G$-graded ring, $M$ a faithful graded multiplication $R$%
-module and $Q$ a graded primary ideal of $R$ such that $M\neq QM$,
then $QM$ is a graded primary-like submodule of $M$.
\end{corollary}
\begin{proof}
Let $rm\in QM$ where $r\in h(R)$ and $m\in h(M).$ So $r\in
Q\subseteq (QM:_{R}M)$ or $m\in $Gr$_{M}(QM)$ by Theorem 4.4.
Therefore, $QM$ is a graded primary-like submodule of $M.$
\end{proof}

\begin{theorem}
Let $R$ be a $G$-graded ring, $M$ a graded multiplication $R$-module
and $N$ a proper graded submodule of $M.$ If $N$ satisfies the
gr-prmeful property, then the following statements are equivalent.
\begin{enumerate}[\upshape (i)]

\item $N$ is a graded primary-like submodule of $M$;

\item $(N:_{R}M)$ is a graded primary ideal of $R$;

\item $N=QM$ for some graded primary ideal $Q$ of $R$ with $Ann(M)=(0:_{R}M) \subseteq Q;$

\item $N$ is graded primary submodule of $M$.
\end{enumerate}

\end{theorem}
\begin{proof}
$(i)\Rightarrow (ii)$By Theorem 4.2.

$(ii)\Leftrightarrow (iii)\Leftrightarrow (iv)$ By \cite[Corollary
2.5]{14}.

$(iii)\Rightarrow (i)$ Let $Q$ be a graded primary ideal of $R$ such that $%
N=QM $ and $Ann(M)\subseteq Q$. Then $M$ is a faithful graded
$R/Ann(M)$-module. The result follows from Corollary 4.5.
\end{proof}
\begin{corollary}
 Let $R$ be a $G$-graded ring and $M$ a graded multiplication $R$-module. If $%
N$ is a graded primary-like submodule of $M$ satisfying the gr-prmeful property, then Gr$(N)$ is a graded prime submodule of
$M$.
\end{corollary}
\begin{proof}
Let $N$ be a graded primary-like submodule of $M.$ By Theorem 4.6, $%
N=QM$ for some graded primary ideal $Q$ of $R$ containing $Ann(M).$
Since $M$ is a graded multiplication module by \cite[Theorem 9]{22},
 Gr$_{M}(QM)=$Gr$(Q)M.$ By \cite[Lemma 1.8]{23} and \cite[Proposition
2.6(ii)]{12}, we conclude that Gr$(Q)M\ $ is a graded prime
submodule. Therefore Gr$(N)$ is a graded prime submodule of $M.$
\end{proof}

\begin{theorem}
Let $R$ be a $G$-graded ring, $M$ a non-zero graded $R$-module and
$N$ a proper graded submodule of $M$. Then the following statements
are equivalent.
\begin{enumerate}[\upshape (i)]

\item $N$ is a graded primary-like submodule of $M;$

\item  $(N:_{R}M)=(N:_{R}U)$ for every graded submodule $U$ of $M$ such that Gr$_{M}(N)\subsetneq U;$
\end{enumerate}

\end{theorem}
\begin{proof}
$(i)\Rightarrow (ii)$ Assume that $N$ is a graded primary-like
submodule of $M$ and let $U$ be any graded submodule of $M$ such that
Gr$_{M}(N)\subsetneq U.$ We show that $(N:_{R}U)=(N:_{R}M).$ Clearly $%
(N:_{R}M)\subseteq (N:_{R}U).$ Let $r=\sum_{g\in G}r_{g}\in
(N:_{R}U).$ Hence for any $g\in G,$ we have $r_{g}U\subseteq N.$ Since
Gr$_{M}(N)\subsetneq U,$ there exists $u_{h}\in h(U)-$Gr$_{M}(N).$
Since $N$ is a graded primary-like submodule of $M,$ $r_{g}u_{h}\in N$ and $%
u_{h}\notin $Gr$_{M}(N)$ for all $g\in G$, we conclude that $r_{g}\in
(N:_{R}M) $ for all $g\in G.$ Hence $r\in (N:_{R}M).$ Therefore $%
(N:_{R}U)=(N:_{R}M).$

$(ii)\Rightarrow (i):$ Suppose that $r_{g}m_{h}\in N$ and $m_{h}\notin $
Gr$_{M}(N)$ for some $r_{g}\in h(R)$ and $m_{h}\in h(M).$ By Lemma
3.1, $U=$Gr$_{M}(N)+Rm_{h}$ is a graded submodule of
$M$. Since $m_{h}\notin $ Gr$_{M}(N),$ we have Gr$_{M}(N)\subsetneq U\subseteq M.$ By (ii), $%
(N:_{R}U)=(N:_{R}M).$ Since $r_{g}m_{h}\in N,$ we have $%
r_{g}(N+Rm_{h})=r_{g}N+Rr_{g}m\subseteq N.$ This yields that $%
r_{g}\in (N:_{R}N+Rm_{h})=(N:_{R}M).$ Therefore $N$ is a graded
primary-like submodule of $M.$
\end{proof}

\begin{theorem}
Let $R$ be a $G$-graded ring and $M,$ $M^{\prime }$ be two graded
$R$-modules and $N^{\prime }$ a graded submodule of $M^{\prime }.$
Let $\varphi :M\rightarrow M^{\prime }$ be a graded epimorphism. If
$N^{\prime }$ is a graded primary-like submodule of $M^{\prime }$,
then$\ \varphi ^{-1}($ $N^{\prime }) $ is a graded primary-like
submodule of $M$.
\end{theorem}
\begin{proof}
Suppose that $r_{g}m_{h}\in \varphi ^{-1}($ $N^{\prime })$ and $r_{g}\notin $ $(\varphi ^{-1}($ $N^{\prime }):_{R}M)$ for some $%
r_{g}\in h(R)$ and $m_{h}\in h(M).$ Hence $\varphi
(r_{g}m_{h})=r_{g}\varphi (m_{h})\in N^{\prime }.$ Since
$r_{g}\notin $ $(\varphi ^{-1}($ $N^{\prime
}):_{R}M),$ we get $r_{g}\notin ($ $N^{\prime }:_{R}M^{\prime })$. Since $%
N^{\prime }$ is a graded primary-like submodule of $M^{\prime }$, $%
r_{g}\varphi (m_{h})\in N^{\prime }$ and $r_{g}\notin ($ $N^{\prime
}:_{R}M^{\prime }),$ we have $\varphi (m_{h})\in $Gr$_{M^{\prime }}($ $%
N^{\prime })$ it follows that $m_{h}\in $ $\varphi ^{-1}($Gr$_{M^{\prime }}($ $%
N^{\prime })).$  By \cite[Theorem 5.3(i)]{11}, we have $\varphi
^{-1}($Gr$_{M^{\prime }}($ $N^{\prime }))=$Gr$_{M}($ $\varphi
^{-1}(N^{\prime })).$ Hence $m_{h}\in $Gr$_{M}(\varphi ^{-1}(N^{\prime
})).$ Therefore $\varphi ^{-1}($ $N^{\prime })$ is a graded
primary-like submodule of $M$.
\end{proof}
\begin{lemma}
Let $R$ be a $G$-graded ring and $M,$ $M^{\prime }$ be two graded $R$%
-modules. Let $\varphi :M\rightarrow M^{\prime }$ be a graded
epimorphism and $N$ a graded submodule of $M$ containing
Ker $\varphi .$ Then $\varphi (G_{M}(N))=G_{M^{\prime }}(\varphi
(N)).$
\end{lemma}
\begin{proof}
The proof is similar to the proof of \cite[Theorem 5.3(ii)]{11}, so
we omit it.
\end{proof}

\begin{theorem}
Let $R$ be a $G$-graded ring and $M,$ $M^{\prime }$ be two graded
$R$-modules. Let $\varphi :M\rightarrow M^{\prime }$ be a graded
epimorphism and $N$ a graded submodule of $M$ containing Ker $\varphi
.$ If $N$ is a graded primary-like submodule of $M$, then $\varphi
(N)\ $is a graded primary-like submodule of $M^{\prime }.$
\end{theorem}
\begin{proof}
Suppose that $rm^{\prime }\in \varphi ($ $N)$ and $r\notin $ $%
(\varphi ($ $N):_{R}M^{\prime })$ for some $r\in h(R)$ and
$m^{\prime }\in h(M^{\prime }).$ It is easy to see\textbf{\
}$r\notin $ $(N:_{R}M)$. Since $\varphi $ is a graded epimorphism,
there exists $m\in h(M)$ such that $\varphi (m)=m^{\prime }$ and
hence $r\varphi (m)=\varphi (rm)\in \varphi ($ $N).$ Thus there
exists $n\in h(N)\ $ such that $\varphi (rm)=\varphi ($ $n)$ this
implies $rm-n\in $Ker$ \varphi \subseteq N.$
This yields that $rm\in N.$ Since $N$ is a graded primary-like submodule of $M,$ $%
rm\in N$ and $r\notin (N:_{R}M)),$ we conclude that $m\in G_{M}(N).$ Hence $%
\varphi (m)=m^{\prime }\in \varphi (G_{M}(N)).$ By Lemma 4.10, we
have $m^{\prime }\in \varphi (G_{M}(N))=G_{M^{\prime }}(\varphi
(N)).$ Therefore $\varphi (N)$ is a graded primary-like submodule of
$M^{\prime }.$
\end{proof}


\begin{lemma}
Let $R$ be a $G$-graded ring and $M$ a graded $R$-module. Let
$S\subseteq h(R)$ be a multiplicatively closed subset of $R$ and $N$
a graded primary-like submodule of $M$ satisfying the gr-prmeful
property such that $S\cap Gr((N:_{R}M))=$ $\emptyset $. Then $%
S^{-1}(N:_{R}M)=(S^{-1}N:_{S^{-1}R}S^{-1}M).$
\end{lemma}
\begin{proof}
Let $\frac{r}{s}\in (S^{-1}N:_{S^{-1}R}S^{-1}M)$\ for some
$\frac{r}{s}\in
h(S^{-1}R).$ Since $N$ satisfies the gr-prmeful property, there exists $%
m_{g}\in h(M)\backslash $Gr$_{M}(N)$. Thus $\ \frac{r}{s}\frac{m_{g}}{1}=\frac{%
rm_{g}}{s}\in S^{-1}N.$ Hence there exists $u\in S$\ such that $%
urm_{g}\in N.$\ Since $N$ is a graded primary-like submodule, $urm_{g}\in N$%
\ and $m_{g}\notin $Gr$_{M}(N),$ we conclude that $ur\in (N:_{R}M).$
By Theorem 4.2, $(N:_{R}M)$ is a graded primary ideal of $R$. Since $%
S\cap $Gr$((N:_{R}M))=$ $\emptyset \ $and $ur\in (N:_{R}M),$ we have
$r\in (N:_{R}M)$ and hence $\frac{r}{s}\in S^{-1}(N:_{R}M).$ So $%
(S^{-1}N:_{S^{-1}R}S^{-1}M)\subseteq S^{-1}(N:_{R}M).$ It is clear that $%
S^{-1}(N:_{R}M)\subseteq (S^{-1}N:_{S^{-1}R}S^{-1}M).$ Thus $%
S^{-1}(N:_{R}M)=(S^{-1}N:_{S^{-1}R}S^{-1}M).$
\end{proof}

The following results study the behavior of graded primary-like
submodules under localization.

\begin{theorem}
Let $R$ be a $G$-graded ring and $M$ a graded $%
R $-module. Let $S\subseteq h(R)$ be a multiplicatively closed
subset of $R$ and $N$ a graded primary-like submodule of $M$
satisfying the gr-primeful property such that $S\cap
$Gr$((N:_{R}M))=$ $\emptyset $. Then $S^{-1}N$ is a
graded primary-like submodule of $S^{-1}R$-module $S^{-1}M$ satisfying the
gr-primeful property.
\end{theorem}
\begin{proof}
It is easy to see that $\frac{m}{1}\in S^{-1}M$ $\backslash S^{-1}N$
for
each $m\in M\backslash N$ and hence $S^{-1}N$ $\neq S^{-1}M.$ Suppose that $%
\frac{r}{s}\frac{m}{l}\in S^{-1}N$ and $\frac{r}{s}\notin
(S^{-1}N:_{S^{-1}R}S^{-1}M)$\ for some $\frac{r}{s}\in h(S^{-1}R)$
and for some $\frac{m}{l}\in h($ $S^{-1}M).$ Since
$S^{-1}((N:_{R}M))\subseteq (S^{-1}N:_{S^{-1}R}S^{-1}M),$ we have
$r\notin (N:_{R}M).$ Thus there exists $k\in S$ \ such that $\
krm\in N.$ By Theorem 4.2, $(N:_{R}M)$ is a
graded primary ideal of $R$. Since $S\cap $Gr$((N:_{R}M))=$ $\emptyset \ $and $%
r\notin (N:_{R}M),$ we conclude that $kr\notin (N:_{R}M)$. So $%
N $ a is graded primary-like submodule of $M$ gives $m\in $Gr$_{M}(N).$ Hence $%
\frac{m}{l}\in S^{-1}($Gr$_{M}(N))\subseteq $Gr$_{M}(S^{-1}N).$ Therefore $%
S^{-1}N$ is a graded primary-like submodule of $\ S^{-1}M.$

Now we show that $S^{-1}N$ satisfies the gr-primeful property. Let
$S^{-1}p$ be a graded prime ideal of $S^{-1}R$ containing $%
(S^{-1}N:_{S^{-1}R}S^{-1}M). $ Hence $p$ $\cap S=\emptyset $ and $%
S^{-1}((N:_{R}M))\subseteq (S^{-1}N:_{S^{-1}R}S^{-1}M)\subseteq
S^{-1}p$.
Then $S^{-1}((N:_{R}M))\subseteq S^{-1}p$ and so easily follows that $%
(N:_{R}M)\subseteq p.$ Since $N$ satisfies the gr-primeful property,
there exists a graded prime submodule $P$ of $M$ containing $N$ such that $%
(P:_{R}M)=p.$ Since $p$ $\cap S=\emptyset $, $S^{-1}P$ is a graded
prime
submodule of $S^{-1}M$ containing $S^{-1}N$ and $%
(S^{-1}P:_{S^{-1}R}S^{-1}M)=S^{-1}p.$ Therefore $S^{-1}N$ satisfies
the gr-primeful property.
\end{proof}


\bigskip\bigskip\bigskip\bigskip

\end{document}